\documentclass[12pt]{amsart}
\usepackage{pgf,tikz,pgfplots} %For TikZ
\pgfplotsset{compat=1.14}
\usepackage{mathrsfs}
\usetikzlibrary{arrows}
\usetikzlibrary{patterns}
\usepackage{pgfplots}
\usetikzlibrary{intersections, pgfplots.fillbetween}
\usepackage[colorlinks=true,urlcolor=blue, citecolor=red,linkcolor=blue,linktocpage,pdfpagelabels, bookmarksnumbered,bookmarksopen]{hyperref}
\usepackage[hyperpageref]{backref}
\usepackage{cleveref}
\usepackage{xcolor}
\usepackage{amsthm} %for citing inside theorem header
\usepackage{latexsym,amsmath,amssymb}
\usepackage{accents}
\usepackage[colorinlistoftodos,prependcaption,textsize=tiny]{todonotes}
\usepackage{a4wide}
\usepackage{soul}
\usepackage{mathtools} %For \chi with a much lower index (\mychi)
\usepackage{xparse} %For a new definiton of \chi with a slightly lower index

\usepackage{enumerate}

\pgfdeclarelayer{ft}
\pgfdeclarelayer{bg}
\pgfsetlayers{bg,main,ft}

\title[Failure of Calder\'on-Zygmund estimates for $A_p$-weights when $p > 2$]{Failure of Calder\'on-Zygmund estimates for degenerate elliptic PDEs with $A_p$-weights when $p > 2$}

\author[Armin Schikorra]{Armin Schikorra}
\address[Armin Schikorra]{Department of Mathematics,
 University of Pittsburgh,
 301 Thackeray Hall,
 Pittsburgh, PA 15260, USA \and Department of Mathematics, 
Chulalongkorn University, 
Bangkok,
Thailand}
  \email{armin@pitt.edu}
\author{Martin Ulmer}
\address[Martin Ulmer]{Department of Mathematics, Brown University, 151 Thayer Street, Providence, RI 02912, USA}
\email{martin\_ulmer@brown.edu}

\definecolor{indigo}{rgb}{0.29, 0.0, 0.51}
\definecolor{p1}{gray}{0.4}
\definecolor{p2}{gray}{0.6}
\definecolor{p3}{gray}{0.98}
\definecolor{p4}{gray}{0.8}
\definecolor{p5}{gray}{0.9}

plus 6pt minus 12pt
\def\eps{\varepsilon}

\def\B{\mathbb{B}}

\def\N{{\mathbb N}}

\renewcommand{\div}{{\rm div}}

\newtheorem{theorem}{Theorem}
\newtheorem{lemma}[theorem]{Lemma}
\newtheorem{corollary}[theorem]{Corollary}
\newtheorem{proposition}[theorem]{Proposition}

\newtheorem{remark}[theorem]{Remark}

\def\diam{{\rm diam\,}}
\def\dist{{\rm dist\,}}

\newcommand{\R}{\mathbb{R}}

\newcommand{\brac}[1]{\left (#1 \right )}

\newcommand{\barint}{
\rule[.036in]{.12in}{.009in}\kern-.16in \displaystyle\int }

\newcommand{\barcal}{\mbox{$ \rule[.036in]{.11in}{.007in}\kern-.128in\int $}}

\def\mvint_#1{\mathchoice
          {\mathop{\vrule width 6pt height 3 pt depth -2.5pt
                  \kern -8pt \intop}\nolimits_{\kern -3pt #1}}%
          {\mathop{\vrule width 5pt height 3 pt depth -2.6pt
                  \kern -6pt \intop}\nolimits_{#1}}%
          {\mathop{\vrule width 5pt height 3 pt depth -2.6pt
                  \kern -6pt \intop}\nolimits_{#1}}%
          {\mathop{\vrule width 5pt height 3 pt depth -2.6pt
                  \kern -6pt \intop}\nolimits_{#1}}}

\numberwithin{theorem}{section} \numberwithin{equation}{section}

\newcommand{\aleq}{\lesssim}
\newcommand{\ageq}{\gtrsim}
\newcommand{\aeq}{\approx}

\let\latexchi\chi
\makeatletter
\renewcommand\chi{\@ifnextchar_\sub@chi\latexchi}
\newcommand{\sub@chi}[2]{% #1 is _, #2 is the subscript
  \@ifnextchar^{\subsup@chi{#2}}{\latexchi_{#2}}%
}
\newcommand{\subsup@chi}[3]{% #1 is the subscript, #2 is ^, #3 is the superscript
  \latexchi_{#1}^{#3}%
}
\makeatother
\makeatletter
\def\tikz@arc@opt[#1]{% over-write!
  {%
    \tikzset{every arc/.try,#1}%
    \pgfkeysgetvalue{/tikz/start angle}\tikz@s
    \pgfkeysgetvalue{/tikz/end angle}\tikz@e
    \pgfkeysgetvalue{/tikz/delta angle}\tikz@d
    \ifx\tikz@s\pgfutil@empty%
      \pgfmathsetmacro\tikz@s{\tikz@e-\tikz@d}
    \else
      \ifx\tikz@e\pgfutil@empty%
        \pgfmathsetmacro\tikz@e{\tikz@s+\tikz@d}
      \fi%
    \fi
    \tikz@arc@moveto
    \xdef\pgf@marshal{\noexpand%
    \tikz@do@arc{\tikz@s}{\tikz@e}
      {\pgfkeysvalueof{/tikz/x radius}}
      {\pgfkeysvalueof{/tikz/y radius}}}%
  }% 
  \pgf@marshal%
  \tikz@arcfinal%
}
\let\tikz@arc@moveto\relax
\def\tikz@arc@movetolineto#1{%
  \def\tikz@arc@moveto{\tikz@@@parse@polar{\tikz@arc@@movetolineto#1}(\tikz@s:\pgfkeysvalueof{/tikz/x radius} and \pgfkeysvalueof{/tikz/y radius})}}
\def\tikz@arc@@movetolineto#1#2{#1{\pgfpointadd{#2}{\tikz@last@position@saved}}}
\tikzset{%
  move to start/.code=\tikz@arc@movetolineto\pgfpathmoveto,%
  line to start/.code=\tikz@arc@movetolineto\pgfpathlineto}
\makeatother
\begin{document}
\begin{abstract}
We use convex integration techniques to provide examples of failure of weighted Calder\'{o}n-Zygmund estimates for degenerate linear elliptic PDEs when the weights are in $A_p$, $p > 2$. 
\end{abstract}

\maketitle
\tableofcontents

\section{Introduction}

We consider solutions $u$ to linear degenerate elliptic PDEs of the form
\[
\div(w\nabla u)=\div(wF)
\]
with a Muckenhoupt weight $w$, and are interested in the classical Calder\'{o}n-Zygmund theory. The purpose of this work is to investigate whether Calderón–Zygmund-type estimates persist for degenerate elliptic equations beyond the natural class of $A_2$ weights, and we show that this is not the case. We work on the simple domain $\mathbb{B}^n$, the unit ball in dimension $n$. Let $A_r=A_r(\mu)$ be the Muckenhoupt class with respect to the Lebesgue measure, and we denote by $w:\mathbb{R}^n\to (0,\infty)$ a weight with characteristic 
\[
[w]_{A_r} := \sup_{Q} \brac{\mvint_{Q} w}\, \brac{\mvint_{Q} w^{-\frac{1}{r-1}}}^{r-1}
\]
where the supremum is taken over cubes $Q \subset \R^n$.

The following is the standard Gehring-Lemma/Meyers version of Calderon-Zygmund theory for weighted elliptic PDEs
\begin{theorem}\label{th:CZpositive}
For any $C_2>0$ there exists $p_0>2$ such that the following is true.

Assume $w\in A_2$ with $[w]_{A_2}\leq C_2$. Then any, say, Lipschitz
$u\in W^{1,\infty}(\B^n)$ solving
\[
 \div(w\nabla u)=\div(wF) \qquad\textrm{in $\B^n$},\\
\]
satisfies
\[
        \int_{\frac12\B^n} w^s |\nabla u|^p
        \leq
        C(C_2,p,s) \brac{
        \int_{\B^n} w^s |F|^p + \|u\|_{L^\infty(\B^n)}^{p} \int_{\B^n} w^s }
\]
for every $p\in[2,p_0]$ and every
\[
        s \in \left \{1,\frac{p}{2} \right \}.
\]
\end{theorem}
Versions of \Cref{th:CZpositive} are well established in the literature for nondegenerate elliptic PDEs, i.e. with weights that are bounded below and above. See for example \cite{Meyer63,DiF96,BW04}, references within and citations thereof. But also in the degenerate case, versions of \Cref{th:CZpositive} are well established, for example in \cite{CUMR18,CMP18,BDGP22}, references within and citations thereof. Furthermore, versions of \Cref{th:CZpositive} can also be found for nonlinear degenerate and nondegenerate elliptic PDEs and systems, see \cite{BWZ07,diFFZ14,CMT19,BDGP22}. We illustrate the proof of \Cref{th:CZpositive} in \Cref{s:Czpositiveproof} for convenience of the reader.

Here, we are more interested in illustrating \emph{how} the estimate in \Cref{th:CZpositive} depends on the weight $w$. In particular, the gain of integrability $p_0>2$ depends on the $[w]_{A_2}$-quantity, and indeed, this dependence is essential. The following is an easy consequence of \cite{AFS08} when combined with the observations in \cite{Armin24,Vincent25} that shows that the higher integrability exponent $p_0$ deteriorates as the size of $[w]_{A_1}$ increases.
\begin{proposition}
For $n=2$ any $p > 2$ and any $\Gamma > 1$ there exists $w \in A_{1}$ with 
\[
 [w]_{A_1} \leq C\, \frac{2}{p-2}
\]
and $u \in W^{1,\infty}_0(\mathbb{B}^n)$ such that $\int_{\mathbb{B}^2} w|\nabla u|^2 < \infty$ and
\[
\begin{cases}
 \div(w \nabla u) = \div(w F) \qquad&\textrm{in $\mathbb{B}^2$}\\
 u = 0 \quad &\text{on $\partial \B^2$}
 \end{cases}
\]
but for any $s >0$
\[
 \int_{\frac{1}{2} \B^n} w^{s} |\nabla u|^p > \Gamma \brac{\int_{\B^n} w^{s} |F|^p  + \|u\|_{L^\infty}^p \int_{\B^n} w^s}
\]
\end{proposition}

Much of work on degenerate elliptic equations has focused on $A_2$ weights, since the foundational work \cite{FKS82} established standard tools for elliptic PDEs like Harnack inequality, Hölder continuity, and weighted Sobolev and Poincar\'{e} inequalities for $A_2$ weights. On the other hand, in harmonic analysis the Calderón–Zygmund theory for singular integral operators aligns naturally with the full scale of Muckenhoupt weights: a classical theorem from \cite{M72} and \cite{CF74} asserts that Calderón–Zygmund operators are bounded on $L^r(w)$ if and only if $w\in A_r$. However, it is far less clear to what extent analogous results hold for solutions to degenerate elliptic equations.

Of course, since any $w \in A_r$ belongs to $A_{r-\eps}$ for some $\eps > 0$, it is natural to ask whether the higher integrability in \Cref{th:CZpositive} persists, possibly with a reduced exponent, for weights in $A_r,r>2$. In other words, we can ask whether for any $A_r$ weight, $r>2$, there is \emph{some} gain of integrability.

Our main theorem shows that this is indeed impossible. Notably, this fails already in one dimension. 
%Indeed, as we show here, for $A_r$, $r> 2$ Calderon-Zygmund theory fails already in $n=1$.
\begin{theorem}\label{th:Baby}
Let $n=1,2$. For any $r >2$ there exists constants $C_r >0$ with the following properties.

For any $\Gamma > 0$ and any $p > 2$ there exists $u,F: \mathbb{B}^n \subset \R^n \to \R$, $u$ Lipschitz and $F$ bounded, such that
\begin{itemize}
 % \item  $u \Big |_{\partial \mathbb{B}^n} = 0$, in particular $u \in W^{1,2}_0(\mathbb{B}^n)$ and $\int_{\mathbb{B}^n} w |\nabla u|^2 < \infty$; and
 \item there exists a piecewise constant weight $w: \mathbb{R}^n\to (0,\infty)$ with $[w]_{A_r} \leq C_r$ for any $r >2$;
 \item the PDE
 \[
  \div (w \nabla u) = \div (w F) \quad \text{in $\mathbb{B}^n$ holds in the distributional sense; but}
 \]
  \item For any $s \in [1,p-1)$ the estimate
 \begin{equation}\label{eq:wronginequality}
  \int_{\frac{1}{2}\mathbb{B}^n} w^{s}|\nabla u|^{p} > \Gamma \brac{\int_{\B^n} w^{s} |F|^p  + \|u\|_{L^\infty}^p \int_{\B^n} w^s}
 \end{equation}
 holds.
 \end{itemize}
 Actually, we can choose $u$ so that $u = 0$ on $\partial \B^n$ and $\|u\|_{L^\infty(\B^n)} \ll 1$ and $w \leq 1$ a.e..
\end{theorem}
Of course, for $p=2$ and $s=1$ \eqref{eq:wronginequality} cannot be true -- from Cacciopoli inequality we immediately get
\[
\int_{\frac{1}{2}\mathbb{B}^n} w|\nabla u|^{2} \leq C \brac{\int_{\B^n} w |F|^2  + \|u\|_{L^\infty}^2 \int_{\B^n} w}.
\]
Let us emphasize here, that the characteristics $[w]_{A_r}$ are bounded by a uniform constant $C_r$ independent of $p$ and $\Gamma$, whereas the corresponding $A_2-$characteristics cannot remain uniformly bounded without contradicting \Cref{th:CZpositive}. This demonstrates that weights $w\in A_r\setminus A_2$ exhibit a fundamentally different behavior: in general, they do not support Calderón–Zygmund-type higher integrability estimates. In this sense, the class $A_2$ is the largest Muckenhoupt class for which such higher integrability results like \Cref{th:CZpositive} can hold.

We provide the proof for Theorem \ref{th:Baby} in $n=1$ and $n=2$ using convex integration techniques. The core idea of the argument is independent of the dimension, and will be presented in the one-dimensional case. However, technical adaptations due to higher dimensions need to be done for all dimensions larger than $1$, and we demonstrate these in the proof for $n=2$. For $n \geq 3$, an adaptation as in \cite{MS24} should be possible and might be necessary. The main new ingredient is the way we combine convex integration corrugations with a control of the coefficient $w$ \emph{at every scale}, by combining elementary splittings with a Whitney cube construction -- which might be of independent interest.

\begin{remark}
It would be interesting to have an example for the homogeneous problem (in exchange for inhomogeneous boundary data), i.e. a weight $w$ and a solution $u$ to
\[
\begin{cases}
 \div(w \nabla u) = 0 \qquad&\textrm{in $\mathbb{B}^n$}\\
 u = x_1 \quad &\text{on $\partial \B^n$}.
 \end{cases}
\]
such that $\int_{\B^n} w |\nabla u|^2 < \infty$ but $\int_{\frac{1}{2}\B^n} w |\nabla u|^q = \infty$ for any $q > 2$.
% It is notable that the example of \Cref{th:Baby} exists also in one dimension, as opposed to the results of \cite{AFS08}.

Observe, however, for $n=1$ it is not clear what to do: If $w \in A_r$ and $u \in W^{1,1}_0((0,1))$ and in distributional sense 
 \[
  (w u')' = 0
 \]
then $w u' \equiv c$ a.e. for some $c \in \R$. 

Then $u'(x) = c w^{-1}(x)$ a.e.. Since $u' \in L^1$ we conclude that $w^{-1} \in L^1$, but then $u(0) = u(1) = 0$ implies $\int c w^{-1}(x) = 0$, i.e. either $c \equiv 0$ or $w \equiv +\infty$ (the latter impossible). So we find that $u' \equiv 0$ and thus $u = 0$ a.e.
\end{remark}

\subsection*{Notation} We use standard notation, constants generally depend on the dimension, other relevant dependencies are indicated. Constants $C$ and $c$ may change from line to line. We write $A \aleq B$ if there exists a constant so that $A \leq C B$, and $A \aeq B$ if $A \aleq B$ and $B \aleq A$.

\subsection*{Outline} In \Cref{sec:Baby1D} we present the proof of Theorem~\ref{th:Baby} in one dimension. In higher dimensions the argument is conceptionally the same, but we apply elementary splitting techniques from convex integration to ensure boundary data, see \Cref{sec:Baby2D}.

\subsection*{Acknowledgment}
We thank Simon Bortz, Hongjie Dong, David Cruz-Uribe, Jill Pipher, Bruno Poggi, Olli Saari for interesting discussions and comments. The algorithm for the construction has been developed with the help of chatgpt.

Armin Schikorra has received funding through the NSF Career DMS-2044898.

\section{Proof of Theorem~\ref{th:Baby} in one dimension}\label{sec:Baby1D}
To begin with, recall the following Whitney cube decomposition
\begin{lemma}\label{la:whitney}
Let $\Omega$ be any open set and assume that 
\[
 \Omega = \bigcup_{i \in \N} W_i \cup \mathcal{N}
\]
where $W_i$ are pairwise disjoint open cubes so that 
\[
 \diam(W) \leq \frac{1}{100} \dist(W, \Omega^c)
\]
and $\mathcal{N}$ is a zeroset.

Then for any cube $C$ with $C \cap \Omega^c \neq \emptyset$ we have 
\[
\sum_{i: C \cap W_i \neq \emptyset} |W_i| \leq \Lambda\, |C|
\]
for a uniform constant $\Lambda = \Lambda(n)$.
\end{lemma}

We first start with several definitions that we are going to need throughout the proof.
For $n \in \N$ set
\[
 b_n := \frac{ \sum_{k=1}^n \frac{1}{k}}{2- \sum_{k=1}^n \frac{1}{k2^k}}, 
\]
\[
 \mu_n = \frac{1}{2n 2^n}, \qquad\textrm{and}
\]
\[
 \beta_n := 1-\sum_{k=1}^n \mu_k=1-\sum_{k=1}^n\frac{1}{2k 2^k}.
\]

Now fix a suitably large $N$ (to be chosen later in dependency of $\Gamma$ and $p$), and set for $j \in \N$
\[
 \alpha_j := \sum_{k=j}^N \mu_k \geq \frac{1}{j}2^{-j-1},
\]
and
\[
 h^j := \brac{\alpha_j}^{-1} \sum_{k=j}^N \mu_k (-2^k) \equiv - \brac{\alpha_j}^{-1} \sum_{k=j}^N \frac{1}{2k}.
\]
In particular 
\[
 h^N = -2^N.
\]

We gather the basic obvious properties:
\begin{lemma}\label{lemma:basicProperties}
 
\begin{enumerate}
 \item For all $n \in \N$ we have $\sum_{k=1}^n \mu_k (-2^k) + \beta_n b_n = 0$ 
 \item $\alpha_1 + \beta_N=1$ and \begin{equation}\label{eq:a1h1bbn}
                                   \alpha_1 h^1 + \beta_N b_N = 0
                                  \end{equation}

 \item $\frac{\alpha_{k+1}}{\alpha_k}+\frac{\mu_{k}}{\alpha_k} =1$ and
 \begin{equation}\label{eq:akhkbbn} h^k = \alpha_k^{-1} \brac{\alpha_{k+1} h^{k+1}+ \mu_{k} (-2^{k})}\end{equation} for all $k=1,\ldots,N-1$
 
\end{enumerate}
\end{lemma}
\begin{proof}
\begin{enumerate}

 \item 
 We have $\mu_k (-2^{k}) = -\frac{1}{2k}$ and $b_n = \frac{\sum_{k=1}^n \frac{1}{k}}{2-2 \sum_{k=1}^n \mu_k} = \frac{\frac{1}{2} \sum_{k=1}^n \frac{1}{k} }{\beta_n}$.
 \item $\alpha_1 + \beta_N=1$ is obvious. We also have $\beta_N b_N =  \frac{1}{2} \sum_{k=1}^N \frac{1}{k}$ so is $ \alpha_1 h^1 + \beta_N b_N = 0$.
 
 \item We have $\alpha_{k+1} + \mu_k = \alpha_{k}$, so $\frac{\alpha_{k+1}}{\alpha_k}+\frac{\mu_{k}}{\alpha_k} =1$ is clear. Moreover,
 \[
   \alpha_k^{-1} \brac{\alpha_{k+1} h^{k+1}+ \mu_{k} (-2^{k})}
   =\alpha_k^{-1} \brac{\sum_{i=k}^N \mu_i (-2^i)}
   =h^k.
 \]

\end{enumerate}
\end{proof}

%\iarmin{observe no error set, magic of $1D$}
\begin{proposition}\label{prop:MainConstructionStep1D}
For any $N\in\mathbb{N}$ and any $\delta > 0$ there exists open sets $\Omega_{B_N}$ and $\Omega_{G_k}$ so that 
\[
 [0,1] = \bigcup_{k=1}^N \Omega_{G_k} \cup \Omega_{B_N} \cup \mathcal{N}
\]
where $\mathcal{N}$ is a $\mathcal{H}^1$ zero set, and a piecewise map $u: [0,1] \to \R$, $u(0) = u(1) = 0$, $\|u\|_{L^\infty} \leq \delta$ and $w: [0,1] \to (0,\infty)$ a piecewise constant function so that
\[
 |\Omega_{B_N}| = \beta_N,
\]
\[
 |\Omega_{G_k}| = \mu_k,
\]
and 
\[
 w = 2^{-k} \quad u' = -2^k \quad \text{a.e. on $\Omega_{G_k}$},
\]
\[
 w = 1 \quad u' = b_N \quad \text{a.e. on $\Omega_{B_N}$}.
\]
If we extend $u$ and $w$ by
\[
 u \equiv 0 \text{ and } w \equiv 1 \quad \text{on $\mathbb{R} \setminus [0,1]$,}
\]
then for any cube $Q \subset\mathbb{R}$ and any $r > 2$ we have
\[
 \mvint_{Q} w \brac{\mvint_{Q} w^{-\frac{1}{r-1}}}^{r-1} \leq C(r)
\]
where $C(r)$ is a constant uniform in $N$.
\end{proposition}
\begin{proof}
By \eqref{eq:a1h1bbn} we can find $u_1$ piecewise affine (the typical sawtooth) with $u_1(0) = u_1(1) = 0$ such that for open, disjoint sets $\Omega_{H^1}$, $\Omega_{B_N}$
\[
 [0,1] = \Omega_{H^1} \cup \Omega_{B_N} \cup \mathcal{N},
\]
where $\mathcal{N}$ is a zeroset,
\[
 |\Omega_{B_N}| = \beta_N, \quad |\Omega_{H^1}| = \alpha_1,
\]
and 
\[
 u_1' = b_N \quad \text{a.e. on $\Omega_{B_N}$},
\]
and 
\[
 u_1' = h^1 \quad \text{a.e. on $\Omega_{H^1}$}.
\]
Increasing the number of oscillations we can also assume that 
\[
\|u_1\|_{L^\infty} \leq 2^{-1}\delta.
\]
We set 
\[
 w := 1 \quad \text{on $\Omega_{B_N}$},
\]
and extend 
\[
 u_1 \equiv 0 \text{ and } w \equiv 1 \quad \text{on $\mathbb{R} \setminus [0,1]=:\Omega^e$}.
\]

Assume now that for $k=1,\ldots,N-1$ we have constructed $u_k$ on $[0,1]$ and have found $\Omega_{H^k}$, $\Omega_{G_1}, \ldots, \Omega_{G_{k-1}}$, $\Omega_{B_N}$ all open and pairwise disjoint sets that cover $[0,1]$ up to a zero set - that satisfy the conclusions of the proposition, and moreover 
\[
 |\Omega_{H^k}| = \alpha_k.
\]

We construct $u_{k+1}$ as follows:

Take a Whitney decomposition (in the sense of \Cref{la:whitney}) of the open set $\Omega_{H^k} = \bigcup_{\ell=1}^\infty \overline{W_{k;\ell}}$ where $W_{k;l}$ are the open, pairwise disjoint Whitney cubes. We can ensure that the Whitney cube are nested as we progress along $k$, i.e. we may assume $W_{k;\ell} \cap W_{k-1;\tilde{\ell}} \neq \emptyset$ implies $W_{k;\ell} \subset W_{k-1;\tilde{\ell}}$.

On each $W_{k;\ell}$ the function $u_{k}$ is affine linear, 
and we have $u_k' \equiv h^k$. By \eqref{eq:akhkbbn} we can replace $u_k$ with a map $u_{k+1}$ (again, the typical sawtooth) so that $u_{k+1}\equiv u_k$ on $\partial W_{k,\ell}$ and find open sets $\tilde{O}_{k;\ell;H^{k+1}}$, $\tilde{O}_{k;\ell;G_{k}}$ such that
\[
W_{k;\ell} = \tilde{O}_{k;\ell;H^{k+1}} \cup \tilde{O}_{k;\ell;G_{k}} \cup \mathcal{N},
\]
where $\mathcal{N}$ is a zeroset, and 
\[
 |\tilde{O}_{k;\ell;H^{k+1}}| = \frac{\alpha_{k+1}}{\alpha_k} |W_{k;\ell}|,
\]
\[
 |\tilde{O}_{k;\ell;G_{k}}| = \frac{\mu_k}{\alpha_k} |W_{k;\ell}|,
\]
and we have 
\[
 u_{k+1}' = -2^k \quad \text{a.e. on $\tilde{O}_{k;\ell;G_{k}}$},
\]
and 
\[
 u_{k+1}' = h^{k+1} \quad \text{a.e. on $\tilde{O}_{k;\ell;H^{k+1}}$}.
\]
We set 
\[
 \Omega_{H^{k+1}} := \bigcup_{\ell} \tilde{O}_{k;\ell;H^{k+1}},
\]
and
\[
 \Omega_{G_{k}} := \bigcup_{\ell} \tilde{O}_{k;\ell;G_{k}}.
\]
Since this $u_{k+1}$ is a sawtooth, again increasing the number of oscillations, we can ensure that 
\[
\|u_{k+1}-u_k\|_{L^\infty} \leq2^{-(k+1)}\delta
\]
Then (for a new zero-set $\mathcal{N}$ which we not relabel)
\[
 \Omega_{H^{k}} = \Omega_{H^{k+1}} \cup \Omega_{G_{k}} \cup \mathcal{N},
\]
and by induction hypothesis
\[
 |\Omega_{G_k}| = \sum_{\ell} |W_{k;\ell}|  \frac{\mu_k}{\alpha_k} = |\Omega_{H^k}|  \frac{\mu_k}{\alpha_k} = \mu_k,
\]
and 
\[
 |\Omega_{H^{k+1}}| = \sum_{\ell} |W_{k;\ell}| \frac{\alpha_{k+1}}{\alpha_k} = |\Omega_{H^k}|  \frac{\alpha_{k+1}}{\alpha_k}  = \alpha_{k+1}.
\]
We set 
\[
 w := 2^{-k} \quad \text{on $\Omega_{G_k}$}.
\]

By induction, we have now defined $w$ and $u:=u_{N}$ on all of $[0,1]$. Since these are finitely many steps, we see that $u$ is Lipschitz, and $w$ is piecewise constant.

Moreover we have (with $u_0 \equiv 0$)
\[
\|u_{N}\|_{L^\infty} \leq \sum_{k=1}^N \|u_{k}-u_{k-1}\|_{L^\infty} \leq \sum_{k=1}^N 2^{-k-1}\delta \leq \delta.
\]

We also observe that 
\begin{equation}
 \label{eq:westonOmegaH}
 0<w \leq 2^{-k} \quad \text{a.e. in $\Omega_{H^k}$}.
\end{equation}

Moreover for any Whitney cube $W_{k;\ell}$ we have, from the nested property of the Whitney cubes, for $j \geq k$
\begin{equation}\label{eq:weq2konwhitney}
 |\{w = 2^{-j}\} \cap W_{k;\ell}| = \prod_{i=k}^{j-1} \frac{\alpha_{i+1}} {\alpha_{i}} \frac{\mu_j}{\alpha_j} |W_{k;\ell}| =  \frac{\mu_j}{\alpha_k} |W_{k;\ell}|.
\end{equation}

Let now $Q$ be an arbitrary cube in $\mathbb{R}$.

\underline{Case 0}: If $Q \subset \Omega_{G^k}$ or $Q \subset \Omega_{B_N} \cup \Omega^e$ then $w$ is constant and thus 
\[
 \mvint_{Q} w \brac{\mvint_{Q} w^{-\frac{1}{r-1}}}^{r-1} =1.
\]

\underline{Case 1}: Assume $Q \cap (\Omega_{B_N} \cup \Omega^e) \neq \emptyset$ and $Q \cap \Omega_{H^1} \neq \emptyset$. Since $(W_{1;\ell})_{\ell \in \N}$ are Whitney cubes of $\Omega_{H^1}$ and $Q \cap \R \setminus \Omega_{H^1} \neq \emptyset$ we have 
\[
 \sum_{\ell \in \N: W_{1;\ell} \cap Q \neq \emptyset} |W_{1;\ell}|\leq 2 |Q|.
\]
We then have 
\[
 \mvint_{Q} w \overset{\eqref{eq:westonOmegaH}}{\leq} 1,
\]
and 
\[
\begin{split}
 \mvint_{Q} w^{-\frac{1}{r-1}} \leq &\frac{|B_{N} \cap Q|}{|Q|}+ \sum_{\ell \in \N: W_{1;\ell} \cap Q \neq \emptyset} \frac{|W_{1;\ell}|}{|Q|} \mvint_{W_{1;\ell}}w^{-\frac{1}{r-1}} \\
\overset{\eqref{eq:weq2konwhitney}}{\leq} &\frac{|B_{N} \cap Q|}{|Q|}+ \sum_{\ell \in \N: W_{1;\ell}\cap Q \neq \emptyset} \frac{|W_{1;\ell}|}{|Q|} \sum_{k=1}^N \frac{\mu_k}{\alpha_1} 2^{k \frac{1}{r-1}}\\  
\leq &1+ \sum_{\ell \in \N: W_{1;\ell}\cap Q \neq \emptyset} \frac{|W_{1;\ell}|}{|Q|} \sum_{k=1}^N \frac{1}{\alpha_1 k 2^{k+1} } 2^{k \frac{1}{r-1}}\\ 
\leq &1+ \sum_{\ell \in \N: W_{1;\ell}\cap Q \neq \emptyset} \frac{|W_{1;\ell}|}{|Q|} \sum_{k=1}^\infty \frac{1}{k} 2^{k (\frac{1}{r-1}-1)}\\
\aleq & \sum_{k=1}^\infty \frac{1}{k} 2^{k (\frac{2-r}{r-1})}.
\end{split}
\]
We used \Cref{la:whitney} in the last step.
We see that the latter is finite (i.e. uniformly bounded in terms of $N$) exactly when $r >2$.

\underline{Case 2}: If $Q \cap B_{N} = \emptyset$, then for some $j \in\{1,\ldots,N-1\}$ we have $Q \subset \Omega_{H^i}$ for $i=0,....j$ and $Q\not \subset \Omega_{H^{j}}$. Thus, $Q \cap \Omega_{G_j}\equiv Q \cap \Omega_{H^j} \setminus \Omega_{H^{j+1}}  \neq \emptyset$.

In this case we have $w \big |_{Q} \leq 2^{-j}$, and thus 
\[
 \mvint_{Q} w \leq 2^{-j}.
\]
Again, we use \Cref{la:whitney}, and get
\[
\begin{split}
 \mvint_{Q} w^{-\frac{1}{r-1}} \leq &\sum_{\ell \in \N: W_{j;\ell} \cap Q \neq \emptyset} \frac{|W_{j;\ell}|}{|Q|} \mvint_{W_{j;\ell}}w^{-\frac{1}{r-1}} \\
\overset{\eqref{eq:weq2konwhitney}}{\leq} &\sum_{\ell \in \N: W_{j;\ell}\cap Q \neq \emptyset} \frac{|W_{j;\ell}|}{|Q|} \sum_{k=j}^N \frac{\mu_k}{\alpha_j} 2^{k \frac{1}{r-1}}\\  
\aleq & \frac{1}{\alpha_j } \sum_{k=j}^N \frac{1}{k} 2^{k (\frac{2-r}{r-1})}\\
\leq & 2j 2^{j} \sum_{k=j}^N \frac{1}{k} 2^{k (\frac{2-r}{r-1})}.\\ 
\end{split}
\]
Thus,
\[
\begin{split}
 &\mvint_{Q} w \brac{\mvint_{Q} w^{-\frac{1}{r-1}}}^{r-1}\\
 \aleq_{r} &2^{-j}  \brac{2^{j \frac{1}{r-1}} \sum_{k=j}^N \frac{j}{k} 2^{(k-j) (\frac{2-r}{r-1})} }^{r-1}\\  
 \aleq_{r} &\brac{\sum_{k=j}^N 2^{(k-j) (\frac{2-r}{r-1})} }^{r-1}\\
 \aleq_{r} &\brac{\sum_{k=1}^\infty 2^{k (\frac{2-r}{r-1})} }^{r-1}.\\
 \end{split}
\]
and again we see that this is bounded uniformly in $N$, if $r >2$.
\end{proof}

\begin{proof}[Proof of Theorem~\ref{th:Baby} in one dimension]
Let $p>2$ and $1\leq s<p-1$, and apply Proposition \ref{prop:MainConstructionStep1D}. Set 
\[
 wF := wu'+1.
\]
Then $\div(w\nabla u)=\div(wF)$, and we have $wF = 0$ a.e. in $[0,1] \setminus \Omega_{B_N}$ and thus 
\[
 \int_{(0,1)} w^s|F|^p =  \beta_N^s |b_N|^p \aleq \brac{\log N}^p.
\]
On the other hand,
\[
\begin{split}
 \int_{(0,1)} w^s |u'|^p =& \beta_N^s |b_N|^p + \sum_{k=1}^N \mu_k 2^{-sk} 2^{pk}\\
 \geq&\sum_{k=1}^N \frac{1}{2 k 2^k} 2^{k(p-s)}
 \aeq \sum_{k=1}^N \frac{1}{k} 2^{k(p-s-1)}. \\
 \end{split}
\]
We see that for $p>2$ and $s<p-1$ we can chose $N$ depending on $\Gamma$ suitably large so that 
\begin{equation}
 \int_{(0,1)} w^s |u'|^p >\Gamma  \int_{(0,1)} w^s|F|^p.\label{eq:1ds=p/2}
\end{equation}

For $p=2$ and $s=1$ this is not possible, since then both sides grow like $(\log N)^p$. Since $\delta$ can be chosen such that $\Vert u\Vert_{L^\infty}\ll 1$, we also have
\[
\|u\|_{L^\infty}^p \int_{(0,1)} w^s\ll 1,
\]
which implies \eqref{eq:wronginequality}.
Lastly, note that the same construction can be done on any interval, in particular on $(-\frac{1}{2},\frac{1}{2})$, and both $u$ and $F$ vanish outside of this interval. This completes the proof of Theorem~\ref{th:Baby} in one dimension.

\end{proof}

\section{Proof of Theorem~\ref{th:Baby} in two dimensions}\label{sec:Baby2D}

We define $b_n,\mu_n,\beta_n,\alpha_n,$ and $h^n$ as in the beginning of Section \ref{sec:Baby1D}, the proof of Theorem~\ref{th:Baby} in one dimension.

We prove a result similar to \Cref{prop:MainConstructionStep1D}. Note here, that additional technical difficulties arise due to the higher dimension $n=2$, namely the existence of the error set $\Omega_{error}$.

\begin{proposition}\label{prop:MainConstructionStep2D}
For any $N\in\mathbb{N}$ and any $\eps>0$ there exists a set $\Omega_{error}$ and open sets $\Omega_{B_N}$ and $\Omega_{G_k}$ so that 
\begin{equation}
 \mathbb{B}^2 = \bigcup_{k=1}^N \Omega_{G_k} \cup \Omega_{B_N} \cup\Omega_{error},\label{eq:B2isUnion}
\end{equation}
and a piecewise map $u: \mathbb{B}^2 \to \R$ with $u|_{\partial\mathbb{B}^2}=0$, $\|u\|_{L^\infty} \leq \eps$ and $w: \mathbb{B}^2 \to (0,\infty)$ a piecewise constant function so that
\begin{equation}\label{eq:measureofOmega_G_k}
\begin{split}
    (\pi-\eps)\beta_N\leq|\Omega_{B_N}|\leq (\pi+\eps) \beta_N,\\
 (\pi-\eps)\mu_k\leq|\Omega_{G_k}|\leq (\pi+\eps)\mu_k,\\
 \int_{\Omega_{error}}1+|\nabla u|^p\leq \eps,
\end{split}
\end{equation}
and 
\begin{equation}\label{eq:wanduongoodset}
\begin{split}
 w = 2^{-k}, \quad \nabla u = \begin{pmatrix}
  -2^{k} \\
  0 
 \end{pmatrix} \quad \text{a.e. on $\Omega_{G_k}$, and}\qquad
 w = 1, \quad \nabla u = \begin{pmatrix}
  b_N \\
  0 
 \end{pmatrix}\quad \text{a.e. on $\Omega_{B_N}$}.
 \end{split}
\end{equation}
If we extend $u$ and $w$ by
\[
 u \equiv 0 \text{ and } w \equiv 1 \quad \text{on $\mathbb{R}^2 \setminus \mathbb{B}^2$,}
\]
then for any cube $Q \subset\mathbb{R}$ and any $r > 2$ we have
\begin{equation}\label{eq:A_rInNewConst}
 \mvint_{Q} w \brac{\mvint_{Q} w^{-\frac{1}{r-1}}}^{r-1} \leq C(r),
\end{equation}
where $C(r)$ is a constant uniform in $N$.
\end{proposition}

\begin{proof}
First, we define the matrices 
\[
 G_n := \begin{pmatrix}
         -2^{n} & 0\\
         0 & 0
        \end{pmatrix},
\]
\[
 H_n := \begin{pmatrix}
         h^n & 0\\
         0 & 0
        \end{pmatrix},
\]
and
\[
 B_N := \begin{pmatrix}
         b_N & 0\\
         0 & 0
        \end{pmatrix}.
\]
Let $(\eps_i)_i$ be a sequence to be determined later.
Due to \eqref{eq:a1h1bbn}, we note that $\alpha_1\delta_{H^1}+\beta_N\delta_{B_N}$ is a laminate of finite order with baricenter
\[\begin{pmatrix}
  0 & 0\\
  0 & 0
 \end{pmatrix}.\]
Hence, we can apply \cite[Lemma 2.1]{KMSX24} and obtain two open sets $\Omega_{H^1}$ and $\Omega_{B_N}$ and a piecewise affine function $u_1:\mathbb{B}^2\to\mathbb{R}^n$ with $u_1|_{\partial\mathbb{B}^2}=0$, $\|u_1\|_{L^\infty} \leq 2^{-1} \eps$, such that
\[
 (1-\eps_1)\beta_N|\mathbb{B}^2|\leq|\Omega_{B_N}|\leq (1+\eps_1)\pi \beta_N,\qquad \textrm{and}\qquad
 (1-\eps_1)\pi\alpha_1\leq|\Omega_{H_1}|\leq (1+\eps_1)\pi\alpha_1,
\]
and 
\[
 \nabla u_1 = \begin{pmatrix}
         b_N \\
          0
        \end{pmatrix} \quad \text{a.e. on $\Omega_{B_N}$},
\]
and 
\[
 \nabla u_1 = \begin{pmatrix}
         h^1 \\
          0
        \end{pmatrix} \quad \text{a.e. on $\Omega_{H^1}$}.
\]
We set 
\[\Omega_{1,error}:=\mathbb{B}^2\setminus(\Omega_{H^1}\cup\Omega_{B_N}),\]
and
\[
 w := 1 \quad \text{on $\Omega_{B_N}\cup\Omega_{1,error}$},
\]
and extend 
\[
 u_1 \equiv 0 \text{ and } w \equiv 1 \quad \text{on $\mathbb{R}^2 \setminus \mathbb{B}^2=:\Omega^e$}.
\]
Also from \cite[Lemma 2.1]{KMSX24}, we obtain that $|\nabla u_1|\leq \max\{b_N, h_1\}\leq 2$, and hence
\[
    \int_{\Omega_{1,error}}1+|\nabla u_1|^p\leq 2^{p+1}\eps_1\pi.
\]

As in the Section \ref{sec:Baby1D}, we would like to proceed with the construction inductively. Assume now that for $k=1,\ldots,N-1$ we have constructed $u_k$ on $\mathbb{B}^2$ and have found $\Omega_{H^k}$, $\Omega_{G_1}, \ldots, \Omega_{G_{k-1}}$, $\Omega_{B_N}$ all open and pairwise disjoint sets such that the conclusion of the proposition holds. Furthermore, let $\Omega_{1,error},...,\Omega_{2,error}$ be sets such that $\mathbb{B}^2 =\Omega_{H^k}\cup \bigcup_{i=1}^{k-1} \Omega_{G_i} \cup \Omega_{B_N} \cup \bigcup_{i=1}^{k}\Omega_{i, error}$, where all of the sets are pairwise disjoint. Assume now that
\begin{equation}
 \prod_{i=1}^k(1-\eps_i)\pi\alpha_k\leq |\Omega_{H^k}| = \prod_{i=1}^k(1+\eps_i)\pi\alpha_k\label{eq:IndHypod=2}
\end{equation}
and
\begin{equation}
    \int_{\Omega_{i,error}}1+|\nabla u_k|^p\leq 2^{i(p+1)}\eps_i\pi\prod_{j=1}^{i-1}(1+\eps_j)\alpha_{i-1}\qquad \textrm{ for all }i=1,...k.\label{eq:IndHypod=2.2}
\end{equation}

As before in the proof of Proposition \ref{prop:MainConstructionStep1D}, take a Whitney decomposition (in the sense of \Cref{la:whitney}) of the open set $\Omega_{H^k} = \bigcup_{\ell=1}^\infty \overline{W_{k;\ell}}$ where $W_{k;l}$ are the open, pairwise disjoint Whitney cubes that are nested as we progress along $k$, i.e. we may assume $W_{k;\ell} \cap W_{k-1;\tilde{\ell}} \neq \emptyset$ implies $W_{k;\ell} \subset W_{k-1;\tilde{\ell}}$.

On each $W_{k;\ell}$ the function $u_{k}$ is affine linear, 
and we have $\nabla u_k \equiv \begin{pmatrix}h^k \\0\end{pmatrix}$. By \eqref{eq:akhkbbn}, $\alpha_{k}^{-1}\Big(\alpha_{k+1}\delta_{H^{k+1}}+\mu_{k+1}\delta_{G_{k}}\Big)$ is a laminate of finite order with baricenter $H^k$, and hence \cite[Lemma 2.1]{KMSX24} allows to replace $u_k$ with a map $u_{k+1}$ so that $u_{k+1} = u_k$ on $\partial W_{k,\ell}$, $\|u_{k_1}-u_k\|_{L^\infty} \leq 2^{-k-1} \eps$, and allows to find disjoint open sets $\tilde{O}_{k;\ell;H^{k+1}}$, $\tilde{O}_{k;\ell;G_{k}}$ such that
\[
W_{k;\ell} = \tilde{O}_{k;\ell;H^{k+1}} \cup \tilde{O}_{k;\ell;G_{k}} \cup \Omega_{k+1,l,error}
\]
with 
\[
 (1-\eps_{k+1})\frac{\mu_k}{\alpha_k}|W_{k,l}|\leq|\tilde{O}_{k;\ell;G_{k}}|\leq (1+\eps_{k+1}) \frac{\mu_k}{\alpha_k}|W_{k,l}|,
\]
and
\[
 (1-\eps_{k+1})\frac{\alpha_{k+1}}{\alpha_k}|W_{k,l}|\leq|\tilde{O}_{k;\ell;H^{k+1}}|\leq (1+\eps_{k+1})\frac{\alpha_{k+1}}{\alpha_k}|W_{k,l}|,
\]
where we set
\[\Omega_{k+1,l,error}:=W_{k,l}\setminus(\tilde{O}_{k;\ell;H^{k+1}}\cup\tilde{O}_{k;\ell;G_{k}}).\]
Furthermore we have 
\[
 \nabla u_{k+1} = \begin{pmatrix} -2^k \\ 0 \end{pmatrix} \quad \text{a.e. on $\tilde{O}_{k;\ell;G_{k}}$},
\]
and 
\[
 \nabla u_{k+1} = \begin{pmatrix} h^{k+1} \\ 0 \end{pmatrix} \quad \text{a.e. on $\tilde{O}_{k;\ell;H^{k+1}}$}.
\]
We set 
\[
 \Omega_{H^{k+1}} := \bigcup_{\ell} \tilde{O}_{k;\ell;H^{k+1}},
\]
\[
 \Omega_{G_{k}} := \bigcup_{\ell} \tilde{O}_{k;\ell;G_{k}},
\]
\[
 \Omega_{k+1,error} := \bigcup_{\ell} \Omega_{k+1,l,error}.
\]
Then
\[
 \Omega_{H^{k}} = \Omega_{H^{k+1}} \cup \Omega_{G_{k}} \cup \Omega_{k+1,error},
\]
and by induction hypothesis \eqref{eq:IndHypod=2}
\[
 \prod_{i=1}^{k+1}(1-\eps_i)\alpha_{k+1}\leq(1-\eps_{k+1})\frac{\alpha_{k+1}}{\alpha_k}|\Omega_{H^k}|=\sum_l(1-\eps_{k+1})\frac{\alpha_{k+1}}{\alpha_k}|W_{k,l}|\leq \sum_l|\tilde{O}_{k;\ell;H^{k+1}}|\leq|\Omega_{H^{k+1}}|,
\]
and
\[
 |\Omega_{H^{k+1}}|\leq \sum_l|\tilde{O}_{k;\ell;H^{k+1}}|\leq\sum_l(1+\eps_{k+1})\frac{\alpha_{k+1}}{\alpha_k}|W_{k,l}| =(1+\eps_{k+1})\frac{\alpha_{k+1}}{\alpha_k}|\Omega_{H^k}|\leq \prod_{i=1}^{k+1}(1+\eps_i)\alpha_{k+1}.
\]
Similarly, 
\[
 \prod_{i=1}^{k+1}(1-\eps_i)\mu_{k}\leq(1-\eps_{k+1})\frac{\mu_k}{\alpha_k}|\Omega_{G_k}|=\sum_l(1-\eps_{k+1})\frac{\mu_k}{\alpha_k}|W_{k,l}|\leq \sum_l|\tilde{O}_{k;\ell;G_{k}}|\leq|\Omega_{G^{k}}|,
\]
and
\[
    |\Omega_{G_k}|\leq
    \sum_l|\tilde{O}_{k;\ell;G_k}| \leq \sum_l(1+\eps_{k+1})\frac{\mu_k}{\alpha_k}|W_{k,l}|= (1+\eps_{k+1})\frac{\mu_k}{\alpha_k}|\Omega_{G_k}|
    \leq\prod_{i=1}^{k+1}(1+\eps_i)\mu_{k}.
\]
Furthermore $|\nabla u_{k+1}|\leq \max\{2^{k},h^{k+1}\}\leq 2^k,$, and hence by \eqref{eq:IndHypod=2}
\[\int_{\Omega_{k+1,error}}1+|\nabla u|^p\leq 2^{(k+1)(p+1)}\eps_{k+1}|\Omega_{H^k}|\leq 2^{(k+1)(p+1)}\eps_{k+1}\pi\prod_{j=1}^{k}(1+\eps_j)\alpha_{k}.\]

We set 
\[
 w := 2^{-k} \quad \text{on $\Omega_{G_k}\cup \Omega_{k+1,error}$}.
\]

Let $\eps>0$ be arbitrary now and let us choose the sequence $(\eps_i)_i$ accordingly so that the conclusion of the proposition holds. Set $\Omega_{error}=\bigcup_{k=1}^N\Omega_{k,error}$. Then \eqref{eq:B2isUnion} holds. Choosing $|\eps_i|\leq 2^{-3N(p+1)}\pi\eps$ and from \eqref{eq:IndHypod=2.2}, we get
\[
\begin{split}
    \int_{\Omega_{error}}1+|\nabla u|^p&\leq \sum_{k=1}^N\int_{\Omega_{k,error}}1+|\nabla u|^p\leq \sum_{k=1}^N 2^{k(p+1)}\eps_i\pi\prod_{j=1}^{k-1}(1+\eps_j)\alpha_{k-1}
    \\
    &\leq N2^{N(p+1)}\pi\big(1+\max\{\eps_1,...,\eps_N\}\big)^{N}\max\{\eps_1,...,\eps_N\}\leq \eps,
\end{split}
\]
and \eqref{eq:measureofOmega_G_k} holds.

By induction, we have now defined $w$ and $u:=u_{N}$ on all of $\mathbb{B}^2$. Since these are finitely many steps, we see that $u$ is Lipschitz, and $w$ is piecewise constant. In addition to that, the construction immediately implies \eqref{eq:wanduongoodset}, and it remains to proof \eqref{eq:A_rInNewConst}.

Moreover for any Whitney cube $W_{k;\ell}$ we have, from the nested property of the Whitney cubes, for $j \geq k$,
\begin{equation}\label{eq:weq2konwhitneyD=2}
 \begin{split}
 |\{w = 2^{-j}\} \cap W_{k;\ell}| = |\{w = 2^{-j}\} \cap W_{k;\ell}\cap(\mathbb{R}^2\setminus\Omega_{error})| + |\{w = 2^{-j}\} \cap W_{k;\ell}\cap\Omega_{error}|
 \\
 \leq\prod_{i=k}^{j-1}(1+\eps_i) \frac{\alpha_{i+1}} {\alpha_{i}} \frac{\mu_j}{\alpha_j} |W_{k;\ell}| + \sum_{i=k}^{j-1}\eps_i|W_{k;\ell}| =  (1+\eps)\frac{\mu_j}{\alpha_k} |W_{k;\ell}|.
 \end{split}
\end{equation}

Now, the proof of \eqref{eq:A_rInNewConst} follows almost verbatim to the proof in Proposition \ref{prop:MainConstructionStep1D}. We only need to replace \eqref{eq:weq2konwhitney} by \eqref{eq:weq2konwhitneyD=2} and carry the additional factor of $(1+\eps)$ from where we apply \eqref{eq:weq2konwhitney}/\eqref{eq:weq2konwhitneyD=2} until the end. This will give rise to the constant $(1+\eps)C(r)$ instead of $C(r)$.

Yet again, by construction
\[
\|u\|_{L^\infty} \leq \sum_{k=1}^N \|u_{k}-u_{k-1}\|_{L^\infty} \leq \sum_{k=1}^\infty 2^{-k-1} \eps \leq \eps.
\]
\end{proof}

\begin{proof}[Proof of Theorem~\ref{th:Baby} in two dimension]

We fix $p>2$ and $1 \leq s<p-1$. Choose the error parameter $\eps>0$ in
\Cref{prop:MainConstructionStep2D} so small that $2^{(N+1)p}\eps\leq 1$.

We apply \Cref{prop:MainConstructionStep2D} and set
\[
        wF:=w\nabla u+\begin{pmatrix}1\\0\end{pmatrix}.
\]
Then $\div(w\nabla u)=\div(wF)$, and $wF\equiv0$ a.e. on each $\Omega_{G_k}$. On $\Omega_{B_N}$ we have
$w=1$ and $F=(b_N+1,0)$, and on $\Omega_{error}$ we have $w\leq 1$ and $|F|\leq 2^N+1$. Hence
\[
\begin{split}
        \int_{\B^2} w^s|F|^p
        &\leq
        |\Omega_{B_N}|\,|b_N+1|^p
        +
        \int_{\Omega_{error}} w^s|F|^p  \\
        &\aleq
        \beta_N |b_N+1|^p
        +
        2^{(N+1)p}|\Omega_{error}|  \\
        &\leq
        (\log N)^p+1.
\end{split}
\]
On the other hand,
\[
\begin{split}
        \int_{\B^2} w^s|\nabla u|^p
        &\geq
        \sum_{k=1}^N
        \int_{\Omega_{G_k}} w^s|\nabla u|^p  \\
        &\geq
        (1-\eps)\sum_{k=1}^N
        \mu_k\,2^{-ks}2^{kp}  \\
        &\ageq
        \sum_{k=1}^N
        \frac1k\,2^{k(p-s-1)} .
\end{split}
\]
Since $s<p-1$, the last sum grows exponentially in $N$, whereas the right-hand
side is bounded by $C(\log N)^p+1$. Choosing $N$ sufficiently large depending on $\Gamma$ gives
\[
        \int_{\B^2} w^s|\nabla u|^p
        >
        \Gamma
        \int_{\B^2} w^s|F|^p .
\]

Since $\|u\|_{L^\infty} \ll 1$, and since we can do the same construction on $\frac{1}{2}\mathbb{B}^2$ instead of $\mathbb{B}^2$, the proof of \Cref{th:Baby} in $n=2$ is completed.
\end{proof}

\appendix

\appendix
\section{Proof of Theorem~\ref{th:CZpositive}}\label{s:Czpositiveproof}
We discuss the proof of \Cref{th:CZpositive}. It is well known and we claim no originality whatsoever. We begin by gathering some basic facts:

\begin{lemma}[$A_r$ weights are doubling]
\label{lem:Ap-doubling}
Let $1<r<\infty$ and let $w\in A_r$. Then $w\,d\mathcal{L}^n$ is a doubling measure.
More precisely, for every cube $Q\subset\mathbb R^n$,
\[
        \int_{2Q} w
        \leq
        2^{nr}[w]_{A_p}\int_Q w .
\]
\end{lemma}

\begin{proof}
Since $Q\subset 2Q$, by H\"older's inequality,
\[
        \frac{|Q|}{|2Q|}
        =
        \frac1{|2Q|}\int_Q w^{1/r}w^{-1/r}
        \leq
        \left(
        \frac1{|2Q|}\int_Q w
        \right)^{1/r}
        \left(
        \frac1{|2Q|}\int_{2Q} w^{-\frac1{r-1}}
        \right)^{\frac{r-1}{r}} .
\]
Using the definition of $[w]_{A_r}$, this gives
\[
        \frac{|Q|}{|2Q|}
        \leq
        [w]_{A_r}^{1/r}
        \left(
        \frac{\int_Q w}{\int_{2Q} w}
        \right)^{1/r}.
\]
The claim follows.
\end{proof}

% \begin{lemma}[Powers of an $A_r$ weight]
% \label{lem:powers-A2}
% Let $w\in A_r$ and let $0\leq\beta\leq1$.  Then
% \[
%         w^\beta\in A_{1+\beta (r-1)}
% \]
% and
% \[
%         [w^\beta]_{1+\beta (r-1)}
%         \leq
%         [w]_{A_r}^{\beta}.
% \]
% \end{lemma}
% \begin{proof}
% The case $\beta=0$ is immediate. Let $0<\beta\leq1$. Then
% \[
% \begin{aligned}
%         [w^\beta]_{A_{1+\beta (r-1)}}
%         &=
%         \sup_Q
%         \left(
%         \frac1{|Q|}\int_Q w^\beta
%         \right)
%         \left(
%         \frac1{|Q|}\int_Q (w^\beta)^{-\frac{1}{ \beta (p-1) }}
%         \right)^{\beta (p-1)}        \\
%         &=
%         \sup_Q
%         \left(
%         \frac1{|Q|}\int_Q w^\beta
%         \right)
%         \left(
%         \frac1{|Q|}\int_Q w^{-\frac{1}{p-1}}
%         \right)^{(p-1)\beta} .
% \end{aligned}
% \]
% By Jensen's inequality, since $\beta \leq 1$
% \[
%         \frac1{|Q|}\int_Q w^\beta
%         \leq
%         \left(
%         \frac1{|Q|}\int_Q w
%         \right)^\beta .
% \]
% This implies the claim.
% \end{proof}

\begin{lemma}[Small powers of an $A_r$ weight]
\label{lem:small-powers-A2}
Let $1<r<\infty$ and let $w:\R^n\to(0,\infty)$ satisfy
\[
        [w]_{A_r}\leq C_r .
\]
Then there exist $\delta=\delta(n,r,C_r)>0$ and
$C=C(n,r,C_r)$ such that for every ball $B$ and every
\[
        -\frac{1+\delta}{r-1}\leq \gamma\leq 1+\delta
\]
one has
\[
        \mvint_B w^\gamma
        \leq
        C
        \brac{\mvint_B w}^{\gamma}.
\]
\end{lemma}
\begin{proof}
The $A_r$ condition says that for every cube $Q$,
\begin{equation}\label{eq:Ar-cube-power-lemma}
        \brac{\mvint_Q w}
        \brac{\mvint_Q w^{-\frac1{r-1}}}^{r-1}
        \leq C_r .
\end{equation}

We first derive the weak reverse Holder inequality for $w$. Since
$t\mapsto t^{-\frac{r-1}{2}}$ is convex on $(0,\infty)$, Jensen's inequality
applied to $w^{-\frac1{r-1}}$ gives
\[
        \brac{\mvint_Q w^{-\frac1{r-1}}}^{-\frac{r-1}{2}}
        \leq
                =
        \mvint_Q w^{1/2} .
\]
Combining this with \eqref{eq:Ar-cube-power-lemma}, we obtain
\begin{equation}\label{eq:weak-RH-w-cube-direct}
        \mvint_Q w
        \leq
        C_r
        \brac{\mvint_Q w^{1/2}}^{2}.
\end{equation}
Thus $w$ satisfies a weak reverse Holder inequality. By Gehring's lemma applied to \eqref{eq:weak-RH-w-cube-direct}, there exist
$\delta_+>0$ and $C_+<\infty$, depending only on $n,r,C_r$, such that
\begin{equation}\label{eq:RH-w-cube-direct}
        \left(
        \mvint_Q w^{1+\delta_+}
        \right)^{\frac1{1+\delta_+}}
        \leq
        C_+
        \mvint_Q w
\end{equation}
for every cube $Q$.

Next we derive the corresponding estimate for $w^{-\frac1{r-1}}$. From
\eqref{eq:Ar-cube-power-lemma},
\[
        \brac{\mvint_Q w^{-\frac1{r-1}}}
        \brac{\mvint_Q w}^{\frac1{r-1}}
        \leq
        C_r^{\frac1{r-1}}.
\]
Since $t\mapsto t^{-\frac1{2(r-1)}}$ is convex on $(0,\infty)$, Jensen's
inequality applied to $w$ gives
\[
        \brac{\mvint_Q w}^{-\frac1{2(r-1)}}
        \leq
        \mvint_Q w^{-\frac1{2(r-1)}} .
\]
Squaring,
\[
        \brac{\mvint_Q w}^{-\frac1{r-1}}
        \leq
        \brac{\mvint_Q w^{-\frac1{2(r-1)}}}^{2}.
\]
Therefore
\begin{equation}\label{eq:weak-RH-dual-cube-direct}
        \mvint_Q w^{-\frac1{r-1}}
        \leq
        C(r,C_r)
        \brac{\mvint_Q w^{-\frac1{2(r-1)}}}^{2}.
\end{equation}
Again by Gehring's lemma, now applied to
\eqref{eq:weak-RH-dual-cube-direct}, there exist $\delta_->0$ and
$C_-<\infty$, depending only on $n,r,C_r$, such that
\begin{equation}\label{eq:RH-dual-cube-direct}
        \left(
        \mvint_Q w^{-\frac{1+\delta_-}{r-1}}
        \right)^{\frac1{1+\delta_-}}
        \leq
        C_-
        \mvint_Q w^{-\frac1{r-1}}
\end{equation}
for every cube $Q$.

Set
\[
        \delta:=\min\{\delta_+,\delta_-\}.
\]

We now prove the power estimate.

First assume $0\leq\gamma\leq1$. Since $t\mapsto t^\gamma$ is concave,
Jensen's inequality gives
\[
        \mvint_Q w^\gamma
        \leq
        \brac{\mvint_Q w}^{\gamma}.
\]

Next assume $1\leq\gamma\leq1+\delta$. By Holder's inequality with exponents
\[
        \frac{1+\delta}{\gamma}
        \qquad\text{and}\qquad
        \frac{1+\delta}{1+\delta-\gamma},
\]
we have
\[
        \mvint_Q w^\gamma
        \leq
        \brac{\mvint_Q w^{1+\delta}}^{\frac{\gamma}{1+\delta}}.
\]
Using \eqref{eq:RH-w-cube-direct},
\[
        \mvint_Q w^\gamma
        \leq
        C
        \brac{\mvint_Q w}^{\gamma}.
\]

Finally assume
\[
        -\frac{1+\delta}{r-1}\leq\gamma\leq0.
\]
Set
\[
        t:=-\gamma(r-1).
\]
Then
\[
        0\leq t\leq1+\delta
\]
and
\[
        w^\gamma
        =
        \brac{w^{-\frac1{r-1}}}^{t}.
\]

If $0\leq t\leq1$, then $a\mapsto a^t$ is concave, so Jensen's inequality
gives
\[
        \mvint_Q w^\gamma
        =
        \mvint_Q
        \brac{w^{-\frac1{r-1}}}^{t}
        \leq
        \brac{\mvint_Q w^{-\frac1{r-1}}}^{t}.
\]

If $1\leq t\leq1+\delta$, then Holder's inequality with exponents
\[
        \frac{1+\delta}{t}
        \qquad\text{and}\qquad
        \frac{1+\delta}{1+\delta-t}
\]
gives
\[
        \mvint_Q
        \brac{w^{-\frac1{r-1}}}^{t}
        \leq
        \brac{
        \mvint_Q
        \brac{w^{-\frac1{r-1}}}^{1+\delta}
        }^{\frac{t}{1+\delta}}.
\]
Using \eqref{eq:RH-dual-cube-direct},
\[
        \mvint_Q w^\gamma
        \leq
        C
        \brac{\mvint_Q w^{-\frac1{r-1}}}^{t}.
\]

Thus, in both cases $0\leq t\leq1+\delta$,
\[
        \mvint_Q w^\gamma
        \leq
        C
        \brac{\mvint_Q w^{-\frac1{r-1}}}^{t}.
\]
By \eqref{eq:Ar-cube-power-lemma},
\[
        \mvint_Q w^{-\frac1{r-1}}
        \leq
        C_r^{\frac1{r-1}}
        \brac{\mvint_Q w}^{-\frac1{r-1}}.
\]
Therefore
\[
        \mvint_Q w^\gamma
        \leq
        C
        \brac{\mvint_Q w}^{-\frac{t}{r-1}}.
\]
Since $t=-\gamma(r-1)$, we have
\[
        -\frac{t}{r-1}=\gamma.
\]
Hence
\[
        \mvint_Q w^\gamma
        \leq
        C
        \brac{\mvint_Q w}^{\gamma}.
\]
This proves the lemma.
\end{proof}

\Cref{lem:small-powers-A2} implies
\begin{corollary}\label{co:strangepowersfinite}
Let $w$ be an $A_2$-weight and $[w]_{A_2} \leq C_2$. Then there exists $\delta > 0$, only depending on $C_2$ and $n$, such that for all $\lambda \in [0,\delta]$
\[
        \left(
        \frac{\int_B w^{1+\lambda}}{\int_B w}
        \right)
        \left(
        \frac{\int_B w^{1-\frac{2\lambda}{p-2}}}{\int_B w}
        \right)^{\frac{p-2}{2}}
        \leq C(C_2),
\]
\end{corollary}
\begin{proof}
Since $w$ is an $A_2$-weight, by \Cref{lem:small-powers-A2}, there exists a $\delta > 0$, depending only on $[w]_{A_2}$ such that 
\begin{equation}\label{eq:reversehoelder}
        \mvint_B w^{1+\lambda} \leq C \brac{\mvint_B w}^{1+\lambda} \quad \forall \lambda \in [0,\delta]
\end{equation}
Moreover we have
\begin{equation}\label{eq:powergamma}
        \frac{\int_B w^{\gamma}}{\int_B w}
        \leq
        C(C_2)\brac{\mvint_B w}^{-\gamma} \quad \forall \gamma \in [-1,1].
\end{equation}

Set
\[
        \alpha:=\frac{2\lambda}{p-2}\in[0,2].
\]
Then
\[
        1-\frac{2\lambda}{p-2}=1-\alpha\in[-1,1].
\]

We have by \eqref{eq:reversehoelder} and \eqref{eq:powergamma}
\[
\begin{split}
        &\frac{\int_B w^{1+\lambda}}{\int_B w}
        \left(
        \frac{\int_B w^{1-\frac{2\lambda}{p-2}}}{\int_B w}
        \right)^{\frac{p-2}{2}}
        \\
        &\qquad\leq
        C(C_2)
        \brac{\mvint_B w}^{\lambda}
        \brac{\mvint_B w}^{-\alpha\frac{p-2}{2}}
        =
        C(C_2)
        \brac{\mvint_B w}^{\lambda-\lambda}
        =
        C(C_2).
\end{split}
\]
\end{proof}

\begin{lemma}[Weighted Poincare inequality]
\label{lem:endpoint-poincare}
Let $[w]_{A_2}\leq C_2$. Then there exist $\theta\in(0,1)$ and
$C=C(n,C_2)$ such that the following two estimates hold for every ball
$B=B(R)$ and every Lipschitz function $v$.

With
\[
        v_{B,w}:=
        \frac{\int_B v\,w}{\int_B w},
\]
one has
% Second,
\[
        \left(
        \frac1{\int_B w}
        \frac{1}{R^2}\int_B
        \left|v-v_{B,w}\right|^2 w
        \right)^{1/2}
        \leq
        C
        \left(
        \frac1{\int_B w}
        \int_B |\nabla v|^{2\theta}w
        \right)^{1/(2\theta)} .
\]
and 
\[
        \left(
        \frac1{|B|} \frac{1}{R^2}
        \int_B
        \left|v-v_{B,w}\right|^2 w
        \right)^{1/2}
        \leq
        C
        \left(
        \frac1{|B|}
        \int_B |\nabla v|^{2\theta}w^\theta
        \right)^{1/(2\theta)} .
\]

\end{lemma}

\begin{proof}
The first inequality is the usual weighted Poincare inequality for $A_2$ weights, cf \cite[Theorem~1.5]{FKS82}: If $w \in A_2$ then there exists $\theta < 1$ such that $[w]_{2\theta} \leq C([w]_{A_2})$, and $\theta$ only depends on $[w]_{A_2}$ as well. Then we apply \cite[Theorem~1.5]{FKS82} for $p= 2\theta$, and we choose $\theta$ sufficiently large (by Jensen that does not change the $[w]_{A_{2\theta}}$-norm) so that for their $k$ we have $k2\theta = 2$. The second estimate follows in a similar fashion but from \cite[Proposition 3]{BDGP22}.

\end{proof}

\begin{corollary}\label{co:cacciopoli}
Let $w: \R^n \to (0,\infty)$ satisfy $[w]_{A_2} \leq C_2$ and assume $u$ solves 
\[
\div (w \nabla u) = \div(w F) \quad \text{in $\Omega$}
\]
then there exists $C=C(n,C_2)$ such that in any ball $B(2R) \subset \Omega$ we have 
\[
\int_{B(R)} w |\nabla u|^2 \leq C\left(  \int_{B(R)} w|F|^2 + R^{-2}  \|u\|_{L^\infty(B(2R))}^2\, \int_{B(2R)} w \right)
\]
and
\[
 \brac{\frac{1}{\int_{B(R)} w} \int_{B(R)} w |\nabla u|^2}^{\frac{1}{2}} 
  \leq C\brac{\frac{1}{\int_{B(2R)} w} \int_{B(2R)} w|F|^2}^{\frac{1}{2}} +  
  C\brac{\frac{1}{\int_{B(2R)} w} \int_{B(2R)} w |\nabla u|^{2\theta}}^{\frac{1}{2\theta}} 
\]
as well as 
\[
 \brac{\mvint_{B(R)} w |\nabla u|^2}^{\frac{1}{2}} 
  \leq C\brac{\mvint_{B(2R)} w|F|^2}^{\frac{1}{2}} +  
  C\brac{ \mvint_{B(2R)} w^\theta |\nabla u|^{2\theta}}^{\frac{1}{2\theta}} 
\]
where $\theta \in (0,1)$ only depends on $C_2$.
\end{corollary}
\begin{proof}
Testing with $\eta^2 (u-c)$ for a constant $c$ and $\eta \in C_c^\infty(B(2R))$, $\eta \equiv 1$ in $B(R)$, $|\nabla \eta| \aleq \frac{1}{R}$ we have the usual Cacciopoli inequality
\[
\int_{B(R)} w |\nabla u|^2 \aleq  \int_{B(R)} w|F|^2 + \frac{1}{R^2} \int_{B(2R)} w |u-c|^2
\]
Dividing everything by $\int_{B(R)} w$ or $|B(R)|$, using that $w$ is doubling, and applying \Cref{lem:endpoint-poincare} we have the claim.
\end{proof}

Combining Gehring Lemma, once with $w dx$ as $|\nabla u|^2$ as the function, and once with $dx$ as the measure and  $w |\nabla u|^2$ as the map,  from \Cref{co:cacciopoli} we find 

\begin{corollary}
Let $w:\R^n\to(0,\infty)$ satisfy $[w]_{A_2}\leq C_2$. Then there exists
$p_0>2$, depending only on $n$ and $C_2$, such that the following holds.

Assume $u$ solves 
\[
\div (w \nabla u) = \div(w F) \quad \text{in $B$}
\]
Then for $s \in \{1,\frac{p}{2}\}$
\[
\begin{split}
        \int_{\frac14 B}|\nabla u|^p w^{s}
        \aleq&
        \int_{B}|F|^p w^{s}
        \\
        &+         |B|^{-\frac{p}{n}}  \|u\|_{L^\infty(B)}^p\, 
        \int_{B}w^{s}  
\end{split}
\]
\end{corollary}
\begin{proof}

The second inequality in \Cref{co:cacciopoli} implies that we can use Gehring lemma to obtain some $p_0>2$ (depending only on $[w]_{A_2}$) and for all $\lambda \in \{0,\frac{p}{2}\}$
\[
\begin{split}
        \int_{\frac14 B}|\nabla u|^p w^{1+\lambda}
        \leq&
        C
        \int_{B}|F|^p w^{1+\lambda}
        \\
        &+
        C
        \left(
        \frac1{\int_{B} w}
        \int_{\frac{1}{2}B}|\nabla u|^2 w
        \right)^{p/2}
        \int_{B}w^{1+\lambda} .
\end{split}
\]
We may assume that $|p_0 - 2| < \delta$ where $\delta>0$ is from \Cref{co:strangepowersfinite} (depending only on $[w]_{A_2}$).

Now apply the first inequality in \Cref{co:cacciopoli} and we have
\[
\begin{split}
        \int_{\frac14 B}|\nabla u|^p w^{1+\lambda}
        \aleq&
        \int_{B}|F|^p w^{1+\lambda}
        \\
        &+
        \left(
        \frac1{\int_{B} w}
        \int_{B} w|F|^2
        \right)^{p/2}
        \int_{B}w^{1+\lambda}\\
        &+ 
        \left(
        \frac1{\int_{B} w}
        R^{-2}  \|u\|_{L^\infty(B)}^2\, \int_{B(2R)} w 
        \right)^{p/2}
        \int_{B}w^{1+\lambda} \\
        \aleq&
        \int_{B}|F|^p w^{1+\lambda}
        \\
        &+
        \brac{\frac1{\int_{B} w}}^{\frac{p}{2}} \int_{B} w^{1+\lambda} |F|^p
        \,
         \brac{\int_{B} w^{1 -\lambda \frac{2}{p-2}}}^{\frac{p-2}{2}} 
        \,
        \int_{B}w^{1+\lambda}\\
        &+ 
        |B|^{-\frac{p}{n}}  \|u\|_{L^\infty(B)}^p\, 
        \int_{B}w^{1+\lambda}  
\end{split}
\]
Hence
\[
\begin{split}
        \int_{\frac14 B}|\nabla u|^p w^{1+\lambda}
        \aleq&
        \int_B |F|^p w^{1+\lambda}
        \\
        &+
        \left(
        \int_B |F|^p w^{1+\lambda}
        \right)
        \left[
        \frac{\int_B w^{1+\lambda}}{\int_B w}
        \left(
        \frac{\int_B w^{1-\frac{2\lambda}{p-2}}}{\int_B w}
        \right)^{\frac{p-2}{2}}
        \right]\\
        &+
        |B|^{-\frac{p}{n}}\|u\|_{L^\infty(B)}^p
        \int_B w^{1+\lambda} .
\end{split}
\]

The square bracket is bounded by a constant depending only on $C_2$, by \Cref{co:strangepowersfinite}.

\end{proof}

\bibliographystyle{abbrv}
\bibliography{bib}

@article {CMP18,
    AUTHOR = {Cao, Dat and Mengesha, Tadele and Phan, Tuoc},
     TITLE = {Weighted-{$W^{1,p}$} estimates for weak solutions of
              degenerate and singular elliptic equations},
   JOURNAL = {Indiana Univ. Math. J.},
  FJOURNAL = {Indiana University Mathematics Journal},
    VOLUME = {67},
      YEAR = {2018},
    NUMBER = {6},
     PAGES = {2225--2277},
      ISSN = {0022-2518,1943-5258},
   MRCLASS = {35J70 (35B45 35B65 35J15 35J25 35J75)},
  MRNUMBER = {3900368},
MRREVIEWER = {Luca\ Capogna},
       DOI = {10.1512/iumj.2018.67.7533},
       URL = {https://doi.org/10.1512/iumj.2018.67.7533},
}

@article{FKS82,
  author  = {Fabes, Eugene B. and Kenig, Carlos E. and Serapioni, Raul P.},
  title   = {The local regularity of solutions of degenerate elliptic equations},
  journal = {Communications in Partial Differential Equations},
  volume  = {7},
  number  = {1},
  pages   = {77--116},
  year    = {1982},
  doi     = {10.1080/03605308208820218}
}

@article {BDGP22,
    AUTHOR = {Balci, Anna Kh. and Diening, Lars and Giova, Raffaella and
              Passarelli di Napoli, Antonia},
     TITLE = {Elliptic equations with degenerate weights},
   JOURNAL = {SIAM J. Math. Anal.},
  FJOURNAL = {SIAM Journal on Mathematical Analysis},
    VOLUME = {54},
      YEAR = {2022},
    NUMBER = {2},
     PAGES = {2373--2412},
      ISSN = {0036-1410,1095-7154},
   MRCLASS = {35B65 (35J70 35R05)},
  MRNUMBER = {4410267},
MRREVIEWER = {Xiumin\ Du},
       DOI = {10.1137/21M1412529},
       URL = {https://doi.org/10.1137/21M1412529},
}

@article {CUMR18,
    AUTHOR = {Cruz-Uribe, David and Martell, Jos\'e{} Mar\'ia and Rios,
              Cristian},
     TITLE = {On the {K}ato problem and extensions for degenerate elliptic
              operators},
   JOURNAL = {Anal. PDE},
  FJOURNAL = {Analysis \& PDE},
    VOLUME = {11},
      YEAR = {2018},
    NUMBER = {3},
     PAGES = {609--660},
      ISSN = {2157-5045,1948-206X},
   MRCLASS = {35J70 (35B45 35J15 35J25 42B20 42B37 47A07 47D06)},
  MRNUMBER = {3738257},
MRREVIEWER = {Luca\ Capogna},
       DOI = {10.2140/apde.2018.11.609},
       URL = {https://doi.org/10.2140/apde.2018.11.609},
}

@ARTICLE{MS24,
       author = {{Mazowiecka}, Katarzyna and {Schikorra}, Armin},
        title = "{A short note on nowhere smooth critical points of polyconvex functionals in arbitrary dimension}",
      journal = {arXiv e-prints},
         year = 2024,
        month = may,
          eid = {arXiv:2405.17084},
        pages = {arXiv:2405.17084},
          doi = {10.48550/arXiv.2405.17084},
archivePrefix = {arXiv},
       eprint = {2405.17084},
 primaryClass = {math.AP},
       adsurl = {https://ui.adsabs.harvard.edu/abs/2024arXiv240517084M}
}

@ARTICLE{Vincent25,
       author = {{Vincent}, Akshara},
        title = "{Non-uniqueness and failure of Calder{\'o}n-Zygmund estimates below the critical exponent for non-monotone PDE with linear growth}",
      journal = {arXiv e-prints},
         year = 2025,
        month = oct,
          eid = {arXiv:2510.20024},
        pages = {arXiv:2510.20024},
          doi = {10.48550/arXiv.2510.20024},
archivePrefix = {arXiv},
       eprint = {2510.20024},
 primaryClass = {math.AP},
       adsurl = {https://ui.adsabs.harvard.edu/abs/2025arXiv251020024V}
}

@ARTICLE{KMSX24,
    AUTHOR = {Kleiner, Bruce and M\"uller, Stefan and Sz\'ekelyhidi, Jr.,
              L\'aszl\'o{} and Xie, Xiangdong},
     TITLE = {Rigidity of {E}uclidean product structure: breakdown for low
              {S}obolev exponents},
   JOURNAL = {Commun. Pure Appl. Anal.},
  FJOURNAL = {Communications on Pure and Applied Analysis},
    VOLUME = {23},
      YEAR = {2024},
    NUMBER = {10},
     PAGES = {1569--1607},
      ISSN = {1534-0392,1553-5258},
   MRCLASS = {35R70 (30C62 30C65 35A35 46E35)},
  MRNUMBER = {4799456},
MRREVIEWER = {Cintia\ Pacchiano Camacho},
       DOI = {10.3934/cpaa.2024029},
       URL = {https://doi.org/10.3934/cpaa.2024029},
}

@article {AFS08,
    AUTHOR = {Astala, Kari and Faraco, Daniel and Sz\'ekelyhidi, Jr.,
              L\'aszl\'o},
     TITLE = {Convex integration and the {$L^p$} theory of elliptic
              equations},
   JOURNAL = {Ann. Sc. Norm. Super. Pisa Cl. Sci. (5)},
  FJOURNAL = {Annali della Scuola Normale Superiore di Pisa. Classe di
              Scienze. Serie V},
    VOLUME = {7},
      YEAR = {2008},
    NUMBER = {1},
     PAGES = {1--50},
      ISSN = {0391-173X,2036-2145},
   MRCLASS = {35J25 (30C62 35B65 35D10)},
  MRNUMBER = {2413671},
MRREVIEWER = {Leonid\ V.\ Kovalev},
}

@article {Armin24,
    AUTHOR = {Schikorra, Armin},
     TITLE = {Failure of {C}alder\'on-{Z}ygmund {E}stimates for the
              p-{L}aplace {E}quation},
   JOURNAL = {Ann. PDE},
  FJOURNAL = {Annals of PDE. Journal Dedicated to the Analysis of Problems
              from Physical Sciences},
    VOLUME = {12},
      YEAR = {2026},
    NUMBER = {1},
     PAGES = {Paper No. 13},
      ISSN = {2524-5317,2199-2576},
   MRCLASS = {35J92 (35B65)},
  MRNUMBER = {5054820},
       DOI = {10.1007/s40818-026-00239-1},
       URL = {https://doi.org/10.1007/s40818-026-00239-1},
}

@article{Meyer63,
     author = {Meyers, Norman G.},
     title = {An $L^p$-estimate for the gradient of solutions of second order elliptic divergence equations},
     journal = {Annali della Scuola Normale Superiore di Pisa - Scienze Fisiche e Matematiche},
     pages = {189--206},
     year = {1963},
     publisher = {Scuola normale superiore},
     volume = {Ser. 3, 17},
     number = {3},
     mrnumber = {159110},
     zbl = {0127.31904},
     language = {en},
     url = {https://www.numdam.org/item/ASNSP_1963_3_17_3_189_0/}
}

@article{diF96,
	title = {Lp {Estimates} for {Divergence} {Form} {Elliptic} {Equations} with {Discontinuous} {Coefficients}},
	volume = {10},
	journal = {Bollettino della Unione Matemàtica Italiana. Serie VII. A},
	author = {Di Fazio, Giuseppe},
	month = jan,
	year = {1996},
}

@article{BW04,
	title = {Elliptic equations with {BMO} coefficients in {Reifenberg} domains},
	volume = {57},
	issn = {1097-0312},
	url = {https://onlinelibrary.wiley.com/doi/abs/10.1002/cpa.20037},
	doi = {10.1002/cpa.20037},
	language = {en},
	number = {10},
	urldate = {2026-03-04},
	journal = {Communications on Pure and Applied Mathematics},
	author = {Byun, Sun-Sig and Wang, Lihe},
	year = {2004},
	note = {\_eprint: https://onlinelibrary.wiley.com/doi/pdf/10.1002/cpa.20037},
	pages = {1283--1310},
}

@article{BWZ07,
	title = {Nonlinear elliptic equations with {BMO} coefficients in {Reifenberg} domains},
	volume = {250},
	issn = {0022-1236},
	url = {https://www.sciencedirect.com/science/article/pii/S0022123607001802},
	doi = {10.1016/j.jfa.2007.04.021},
	number = {1},
	urldate = {2026-05-13},
	journal = {Journal of Functional Analysis},
	author = {Byun, Sun-Sig and Wang, Lihe and Zhou, Shulin},
	month = sep,
	year = {2007},
	pages = {167--196},
}

@article{CMT19,
  title={Gradient estimates for weak solutions of
 linear elliptic systems with singular-degenerate
 coefficients},
  author={Dat Cao and Tadele Mengesha and Tuoc Van Phan},
  journal={Nonlinear Dispersive Waves and Fluids},
  year={2019},
  url={https://api.semanticscholar.org/CorpusID:53632385}
}

@article{diFFZ14,
author = {Di Fazio, Giuseppe and Fanciullo, Maria and Pietro, Zamboni},
year = {2014},
month = {09},
pages = {907-917},
title = {Lp estimates for degenerate elliptic systems with VMO coefficients},
volume = {25},
journal = {St. Petersburg Mathematical Journal},
doi = {10.1090/S1061-0022-2014-01322-2}
}

@article{M72,
	title = {Weighted norm inequalities for the {Hardy} maximal function},
	volume = {165},
	issn = {0002-9947, 1088-6850},
	url = {https://www.ams.org/tran/1972-165-00/S0002-9947-1972-0293384-6/},
	doi = {10.1090/S0002-9947-1972-0293384-6},
	language = {en},
	urldate = {2024-03-26},
	journal = {Transactions of the American Mathematical Society},
	author = {Muckenhoupt, Benjamin},
	year = {1972},
	pages = {207--226},
}

@article{CF74,
	title = {Weighted norm inequalities for maximal functions and singular integrals},
	volume = {51},
	issn = {0039-3223, 1730-6337},
	url = {https://www.impan.pl/en/publishing-house/journals-and-series/studia-mathematica/all/51/3/100345/weighted-norm-inequalities-for-maximal-functions-and-singular-integrals},
	doi = {10.4064/sm-51-3-241-250},
	language = {en},
	urldate = {2020-12-08},
	journal = {Studia Mathematica},
	publisher = {Instytut Matematyczny Polskiej Akademii Nauk},
	author = {Coifman, R. and Fefferman, C.},
	year = {1974},
	pages = {241--250},
}

\end{document}